\documentclass[12pt, reqno, a4paper]{amsart}
\usepackage[a4paper, total={6in, 10in}]{geometry}
\usepackage[T1]{fontenc}
\usepackage[utf8]{inputenc}

\usepackage{todonotes}

\usepackage{amssymb}   
\usepackage{amsthm}      
\usepackage{thmtools}     
\usepackage{mathtools}   
\usepackage{mathrsfs}     
\usepackage{physics}       

\usepackage[backend=biber,style=numeric-comp, sorting=nyt, maxnames=99]{biblatex}
\usepackage[hidelinks, hypertexnames=false]{hyperref}
\usepackage[nameinlink, noabbrev, capitalize]{cleveref}
\usepackage{bookmark}

\declaretheorem[style = plain, numberwithin=section]{theorem}
\declaretheorem[style = plain,      sibling = theorem]{corollary}

\declaretheorem[style = plain,      sibling = theorem]{proposition}
\declaretheorem[style = definition, sibling = theorem]{definition}
\declaretheorem[style = definition, sibling = theorem, qed=$\spadesuit$]{example}
\declaretheorem[style = remark,    numbered = no]{remark}
\declaretheorem[style = definition, sibling = theorem]{question}

\usepackage[scaled]{beramono}  
\usepackage{listings}
\usepackage{xcolor}

\definecolor{codegreen}{rgb}{0,0.6,0}
\definecolor{codegray}{rgb}{0.5,0.5,0.5}
\definecolor{backcolour}{rgb}{0.95,0.95,0.92}
\definecolor{codegreen}{rgb}{0,  0.6,  0.35}

\lstdefinestyle{mystyle}{
	backgroundcolor=\color{backcolour},   
	commentstyle=\color{codegreen},
	keywordstyle=\color{blue},
	numberstyle=\tiny\color{codegray},
	stringstyle=\color{codegreen},
	basicstyle=\ttfamily\footnotesize,
	breakatwhitespace=false,         
	breaklines=true,                 
	captionpos=b,                    
	keepspaces=true,                 
	numbers=left,                    
	numbersep=5pt,                  
	showspaces=false,                
	showstringspaces=false,
	showtabs=false,                  
	tabsize=2
}
\usepackage[shortlabels]{enumitem}
\setenumerate{label=(\roman{enumi})}

\newcommand{\N}{\mathbb{N}}   
\newcommand{\Z}{\mathbb{Z}}   
\newcommand{\R}{\mathbb{R}}   
\newcommand{\C}{\mathbb{C}}   
\newcommand{\aut}{\mathrm{Aut}}
\renewcommand{\L}{\mathcal{L}}
\newcommand{\U}{\mathcal{U}}
\newcommand{\I}{\mathcal{I}}
\newcommand{\dotimes}{\overset{\cdot}{\otimes}}

\newcommand{\B}{\mathcal{B}}      

\usepackage{graphicx}                                                         

\title{On Fourier and Fourier-Stieltjes Algebras of C*-Dynamical Systems}

\author{Alexander G. Ravnanger}
\address{Department of Mathematical Sciences, 
University of Copenhagen, 
Universitetsparken 5,
2100 Copenhagen}
\email[]{agr@math.ku.dk}
	
\begin{document}
		\begin{abstract}
We continue the study of the Fourier-Stieltjes algebra of a C*-dynamical system, initiated by Bédos and Conti, and recently extended by Buss, Kwa{\'s}niewski, McKee and Skalski. Firstly, we introduce and study a natural notion of a Fourier algebra of a C*-dynamical system. Notably, we show that it can equivalently be defined either as the closure of the multipliers with finite support or as the closure of multipliers coming from regular equivariant representations. Secondly, we undertake an analysis of the equivariant representation theory of commutative systems. Our main result about this is a description of the group theoretical aspect of the equivariant representation theory in terms of cocycle representations of the underlying transformation group. 
		\end{abstract}
	
\maketitle

\section{Introduction}

The Fourier-Stieltjes algebra of a locally compact group was introduced by Eymard in \cite{eymard}, with the objective of defining a function algebra on the group that could serve as a dual object also to non-abelian groups in lieu of the Pontryagin dual. The Fourier-Stieltjes algebra consists of coefficient functions of unitary representations of the group, but it can also be defined as the span of the positive definite functions on the group. This set admits a norm with respect to which it is isometrically isomorphic to the dual of the universal group C*-algebra of the underlying group. Inside the Fourier-Stieltjes algebra, resides a distinguished ideal, called the Fourier algebra. The Fourier algebra is the closure inside the Fourier-Stieltjes algebra of the functions with compact support, equivalently the closure of the coefficient functions of the left-regular representation. The Fourier algebra can also be identified with the predual of the group von Neumann algebra associated to the group. In the abelian case, the Fourier-Stieltjes algebra is isometrically isomorphic to the convolution algebra of Radon measures on the Pontryagin dual of the group via the Fourier-Stieltjes transform. By identifying the algebra of integrable functions with the measures that are absolutely continuous with respect to the Haar measure, the Fourier-Stieltjes transform restricts to an isometric isomorphism between the Fourier algebra and the algebra of integrable functions on the dual, the Fourier transform. 

Both the Fourier and the Fourier-Stieltjes algebra contain ample analytic information about the underlying group. For example, the elements of the Fourier-Stieltjes algebra induce completely bounded maps of the universal and reduced group C*-algebras. Even more interestingly, it was shown by Walter in \cite{walter72} that both the Fourier and Fourier-Stieltjes algebras are perfect invariants of the group in the sense that if there is an isometric isomorphism between the Fourier/Fourier-Stieltjes algebras of two locally compact groups, those groups are topologically isomorphic. Building on Walter's work, Arendt and de Canniére showed in \cite{arendt1983} that one can discard the norm on the Fourier-Stieltjes algebra at the cost of remembering a double order structure on it and maintain a perfect invariant. This double order structure consists of both the order induced by the cone of positive definite functions inside the Fourier-Stieltjes algebra and the order coming from pointwise comparison. 

In a series of papers, most notable for our purposes \cite{BC2012,BC2015}, Bédos and Conti have studied Fourier theory in the framework of twisted discrete unital C*-dynamical systems. Having studied the class of equivariant representations of such systems, whose coefficient functions induce completely bounded maps of the associated crossed products, they coined a definition of the Fourier-Stieltjes algebra of a C*-dynamical system in \cite{BC2016}. They showed that it can be realized as the span of an approriate class of multipliers of the system, generalizing positive definite functions on a group. In a recent preprint \cite{skalski}, Buss, Kwa{\'s}niewski, McKee and Skalski generalize their theory to the context of twisted actions by étale groupoids on C*-bundles. 

With the results of Walter, Arendt and de Canniére in mind, it is natural to ask how rigidly the Fourier-Stieltjes algebra, considered with appropriate structural data, determines the underlying C*-dynamical system. These questions were first adressed in a subsequent paper \cite{BC2021} by Bédos and Conti. In the present article, we first propose to define the Fourier algebra of a unital discrete C*-dynamical system as the closure inside the Fourier-Stieltjes algebra of the multipliers with finite support. We then show that this algebra coincides with the closure of the multipliers that arise as coefficient functions of regular equivariant representations. Secondly, we provide a study of the equivariant representation theory of commutative systems. Our main result in this regard is that the group theoretical aspect of this representation theory can be formulated in terms of what we call cocycle representations of the underlying group action on Hilbert bundles, building on a concept previously studied for measure spaces. Using these descriptions, we are able to compute the Fourier-Stieltjes algebra in some simple examples, which provide some insight into questions of rigidity and which data may be needed in order to obtain a fine invariant for C*-dynamical systems out of the Fourier-Stieltjes algebra. In addition to their relevance to the study of the Fourier-Stieltjes algebras, these results may be useful in order to construct C*-correspondences over crossed products coming from actions on commutative C*-algebras. \\

\textbf{Outline.} The paper is organized as follows: \cref{chap_prelim} contains some background theory the reader should be familiar with. In \cref{chap_fourier}, we introduce the Fourier algebra of a discrete unital C*-dynamical system and study some of its properties. \cref{chap_comm} contains our analysis of the representation theory of commutative systems. \\

\section{Preliminaries}
\label{chap_prelim}

Throughout this paper, all groups will be discrete. Moreover, all C*-algebras will be unital and a $*$-homomorphism between C*-algebras is always assumed to be unital. By a representation of a C*-algebra, we always mean a $*$-representation. Inner products are linear in the second variable and antilinear in the first. Throughout, $A$ will denote a C*-algebra with unit $1_A$ and $\Gamma$ a discrete group with neutral element $e$. \\

\textbf{C*-dynamical systems and crossed products.} A \textit{ (unital) C*-dynamical system} is a triple $\Sigma = (A, \Gamma, \alpha)$ consisting of a unital C*-algebra $A$ and a discrete group $\Gamma$ acting on $A$ by $*$-automorphisms according to a group homomorphism $\alpha \colon \Gamma \to \aut(A)$. It follows from Gelfand duality that if $\alpha$ is an action of a group $\Gamma$ on a commutative C*-algebra with Gelfand spectrum $\Omega$, there is a uniquely determined action of $\Gamma$ on $\Omega$ such that the action on the C*-algebra is given by $$\alpha_g(f)(x) = f(g^{-1}x) \quad (g \in \Gamma, f \in C(\Omega), x \in \Omega),$$ see for example \cite[Proposition 2.7]{williams}. If $\Gamma$ is a group acting on a compact Hausdorff space $\Omega$, we call the pair $(\Gamma, \Omega)$ a \textit{transformation group}. Thus, there is a perfect correspondence between transformation groups and commutative C*-dynamical systems. 

The space $C_c(\Gamma, A)$ of $A$-valued functions on $\Gamma$ with finite support becomes a unital $*$-algebra with multiplication and involution given by 
\begin{align*}
	(f_1 * f_2)(g) &= \sum_{h \in \Gamma} f_1(h) \alpha_h(f_2(h^{-1}g)) \\
	f_1^*(g) &= \alpha_g(f_1(g^{-1}))^*
\end{align*} for $f_1, f_2 \in C_c(\Gamma, A)$ and $g \in \Gamma$. A \textit{covariant representation} of $\Sigma$ on a Hilbert space $H$ is a pair $(\pi, u)$ consisting of a $*$-homomorphism $\pi \colon A \to \B(H)$ and a group homomorphism $u \colon \Gamma \to \U(H)$ into the group of unitary operators on $H$ such that $$\pi (\alpha_g(a)) = u(g) \pi(a) u(g)^* \quad (g \in \Gamma, a \in A).$$ For $f \in C_c(\Gamma, A)$, set $$\norm{f} = \sup \left\{ \norm{ \sum_{g \in \Gamma} \pi(f(g)) u(g)} \right\},$$ where the supremum is taken over all covariant representations of $\Sigma$. This defines a C*-norm on $C_c(\Gamma, A)$. The \textit{full crossed product} associated to $\Sigma$ is the completion of $C_c(\Gamma, A)$ with respect to this norm, and we denote it by $C^*(\Sigma)$. 

Given a representation $\pi \colon A \to \B(H)$ of $A$ on a Hilbert space $H$, one can construct a \textit{regular covariant representation} $(\tilde{\pi}, \tilde{\lambda}^H)$ of $\Sigma$ on the direct sum $H^\Gamma := \bigoplus_{g \in \Gamma} H$ as follows: 
\begin{align*}
	\left( \tilde{\pi}(a) \xi \right) (h) &= \pi(\alpha_g^{-1}(a)) \xi(h) \\
	\left( \tilde{\lambda}^H (g) \xi \right)(h) &= \xi(g^{-1}h) 
\end{align*} 
for $a \in A, g, h \in \Gamma$ and $\xi \in H^\Gamma$. The integrated form $\Lambda = \tilde{\pi} \times \tilde{\lambda}^H$ is the representation of $C^*(\Sigma)$ on $H^\Gamma$ that is uniquely determined by the formula $$\Lambda (f) = \sum_{g \in \Gamma} \tilde{\pi}(f(g)) \tilde{\lambda}^H(g) \quad (f \in C_c(\Gamma, A)).$$ By picking any faithful representation $\pi \colon A \to \B(H)$ of $A$, one may define the \textit{reduced crossed product $C_{\mathrm{r}}^*(\Sigma)$ associated to $\Sigma$} as the image of $C^*(\Sigma)$ under $\Lambda$, see \cite[Proposition 4.1.5]{brown}. More on crossed products can be found in \cite{williams}. \\

\textbf{Hilbert C*-modules.} We follow the conventions of \cite{lance}, whereto we also refer the reader for more information about Hilbert C*-modules. 

Let $A$ be a C*-algebra. An \textit{inner product $A$-module} is a complex vector space $X$ with a right $A$-module structure that is compatible with the scalar multiplication in the sense that $$\lambda(\xi \cdot a) = (\lambda \xi) \cdot a  = \xi \cdot (\lambda a) \quad (\lambda \in \C, \xi \in X, a \in A),$$ and a map $\langle \cdot , \cdot \rangle \colon A \times A \to A$ such that 
\begin{enumerate}
	\item $\langle \xi, \alpha \eta + \beta \zeta \rangle = \alpha \langle \xi, \eta \rangle + \beta \langle \xi, \zeta \rangle$ for all $\xi, \eta, \zeta \in X$ and $\alpha, \beta \in \C$; 
	\item $\langle \xi, \eta \cdot a \rangle = \langle \xi, \eta \rangle a$ for all $\xi, \eta \in X$ and $a \in A$; 
	\item $\langle \eta, \xi \rangle = \langle \xi, \eta \rangle^* $ for all $\xi, \eta \in X$; 
	\item $\langle \xi, \xi \rangle \geq 0$ with equality if and only if $\xi = 0$ for all $\xi \in X$. 
\end{enumerate} If $X$ is an inner product $A$-module, there is a norm on $X$ given by $$\norm{\xi} = \norm{ \langle \xi, \xi, \rangle }^{\frac{1}{2}} \quad (\xi \in X).$$ If $X$ is complete with respect to this norm, $X$ is called a \textit{Hilbert $A$-module}. We denote by $\L(X)$ the C*-algebra of adjointable operators on $X$. Moreover, we denote by $\I(X)$ the group of invertible bounded $\C$-linear maps of $X$. An adjointable map $U \colon X \to Y$ between Hilbert $A$-modules $X$ and $Y$ is called \textit{unitary} if $UU^* = \mathrm{id}_Y$ and $U^*U = \mathrm{id}_X$. A map is unitary if and only if it is a surjective $A$-linear isometry, see \cite[Theorem 3.5]{lance}.

Let $B$ also be a C*-algebra. Given a Hilbert $A$-module $X$ and a Hilbert $B$-module $Y$ along with a representation $\pi \colon A \to \L(Y)$ of $A$ by adjointable operators on $Y$, there is a way of constructing a tensor product of $X$ and $Y$: Let $X \otimes_{\mathrm{alg}} Y$ denote the tensor product of $X$ and $Y$ as vector spaces. Let $$X \otimes_A Y = \frac{X \otimes_{\mathrm{alg}} Y
}{\mathrm{span} \{ xa \otimes y - x \otimes \pi(a) y : x \in X, y \in Y, a \in A \}}.$$ For $x \in X$ and $y \in Y$, we denote by $x \dotimes y$ the image of $x \otimes y \in X \otimes_{\mathrm{alg}} Y$ in $X \otimes_A Y$. The space $X \otimes_A Y$ can be given the structure of an inner product $B$-module such that 
\begin{align*}
	(x_1 \dotimes y_1) \cdot b &= x \dotimes (y_1 \cdot b) \\
	\langle x_1 \dotimes y_1, x_2 \dotimes y_2 \rangle &= \langle y_1, \pi( \langle x_1, x_2 \rangle ) y_2 \rangle 
\end{align*} 
for $x_1, x_2 \in X, y_1, y_2 \in Y$ and $b \in B$. The completion of $X \otimes_A Y$ with respect to this inner product is called the \textit{internal tensor product of $X$ and $Y$ with respect to $\pi$} and it is denoted by $X \otimes_\pi Y$. For more details, see Chapter 4 of \cite{lance}. \\

\textbf{Equivariant representations and multipliers.} Equivariant representations of twisted C*-dynamical systems were introduced by Bédos and Conti in \cite{BC2012} and further studied in \cite{BC2015}, see also \cite[Section 5]{skalski}. Throughout this section, let $\Sigma = (A, \Gamma, \alpha)$ be a C*-dynamical system. 

\begin{definition}[{\cite[Definition 4.2]{BC2012}}]
	\label{def:equiv} 
	An \textit{equivariant representation} of $\Sigma$ on a Hilbert $A$-module $X$ is a pair $(\rho, v)$ consisting of a $*$-homomorphism $\rho \colon A \to \L(X)$ and a group homomorphism $v \colon \Gamma \to \I(X)$ such that the equations \begin{enumerate}
		\item $\rho(\alpha_g(a)) = v(g) \rho(a) v(g)^{-1}$;
		\item $\alpha_g( \langle \xi, \eta \rangle ) = \langle v(g) \xi, v(g) \eta \rangle $;
		\item $v(g) (\xi \cdot a) = (v(g) \xi) \cdot \alpha_g(a)$
	\end{enumerate}
	hold for all $g \in \Gamma, a \in A$ and $\xi, \eta  \in X$. 
\end{definition}

\begin{remark}
	Notice that it follows from (ii) in the definition above that if $v$ is the group representation in an equivariant representation, $v$ is an isometric representation. Indeed, fix $g \in \Gamma$ and $\xi \in X$. Using (ii) and the fact that $*$-isomorphisms are isometric, one sees that $$ \norm{ v(g) \xi }^2 = \norm{ \langle v(g) \xi, v(g) \xi \rangle }^2 = \norm{ \alpha_g (\langle \xi, \xi \rangle) }^2 = \norm{ \langle \xi, \xi \rangle}^2 = \norm{\xi}^2.$$ 
\end{remark}

\begin{example}[{\cite[Example 4.6]{BC2012}}]
	The \textit{trivial equivariant representation} of $\Sigma$ is the pair $(\ell, \alpha)$, where $\ell \colon A \to \L(A)$ is given by $\ell(a)b = ab$ for $a, b \in A$. 
\end{example}

The next example explains a construction on equivariant representations that will play an important role later in this paper. 

\begin{example}[{\cite[Examples 4.6 and 4.7]{BC2012}}]
	\label{ex:reg_equiv}
	Given an equivariant representation $(\rho, v)$ of $\Sigma$ on a Hilbert $A$-module $X$, one can construct another equivariant representation $(\check{\rho}, \check{v})$ of $\Sigma$ on the direct sum $X^\Gamma := \bigoplus_{g \in \Gamma} X$, called the \textit{associated regular equivariant representation}, as follows: \begin{align*}
		\left( \check{\rho}(a) \xi \right)(h) &= \rho(a) \xi(h) \\
		\left( \check{v}(g) \xi \right)(h) &= v(g) \xi(g^{-1}h)
	\end{align*}
	for $a \in A$ and $g, h \in \Gamma$. The regular equivariant representation associated to the trivial one is referred to as the \textit{regular equivariant representation}. Regular equivariant representations generalize the left-regular representation of a group. 
\end{example}

Note that one can take direct sums and tensor products of equivariant representations: Given a family $\{ (\rho_i, v_i) \}_{i \in I}$ of equivariant representations of $\Sigma$ on Hilbert $A$-modules $X_i$ for $i \in I$, there exists an equivariant representation $\left( \bigoplus_i \rho_i, \bigoplus_i v_i \right)$ of $\Sigma$ on $\bigoplus_i X_i$ such that $$\left( \bigoplus_i \rho_i \right)(a) = \bigoplus_i \rho(a) \quad \text{and} \quad \left( \bigoplus_i v_i \right) (g) = \bigoplus_i v_i(g)$$ for each $a \in A$ and $g \in \Gamma$. Furthermore, if $(\rho_1, v_1)$ and $(\rho_2, v_2)$ are equivariant representations of $\Sigma$ on Hilbert $A$-modules $X_1$ and $X_2$ respectively, there exists an equivariant representation $(\rho_1 \otimes \rho_2, v_1 \otimes v_2)$ of $\Sigma$ on the internal tensor product $X_1 \otimes_{\rho_2} X_2$ such that 
\begin{align*}
	(\rho_1 \otimes \rho_2)(a)(x_1 \dotimes x_2) &= \rho_1(a) x_1 \dotimes x_2 \\
	(v_1 \otimes v_2)(g) (x_1 \dotimes x_2) &= v_1(g) x_1 \dotimes v_2(g) x_2 
\end{align*} 
for all $a \in A, x_1 \in X_1, x_2 \in X_2$ and $g \in \Gamma$. For a more in-depth explanation of these constructions, see \cite[Section 2.2]{BC2016}.

In \cite[Section 5]{BC2016}, it is explained how equivariant representations of $\Sigma$ give rise to C*-correspondences over $C^*(\Sigma)$ and $C^*_{\mathrm{r}}(\Sigma)$. 

A \textit{multiplier} of $\Sigma$ is a map from $\Gamma \times A$ to $A$ that is linear in the second variable. It is sometimes notationally convenient to denote the value of a multiplier $T$ at $(g, a) \in \Gamma \times A$ by $T_g(a)$ rather than $T(g, a)$ and treat $T_g = T(g, \cdot) \colon A \to A$ as a linear self-map of $A$. The set $L(\Gamma, A)$ of multipliers is a unital algebra with $\Gamma$-pointwise addition and composition. Given a multiplier $T$ and an $A$-valued function $f \in C_c(\Gamma, A)$ with finite support, we denote by $T \cdot f$ the $A$-valued function given by $$(T \cdot f)(g) = T_g(f(g)) \quad (g \in \Gamma).$$ 

\begin{definition}[{\cite[Definition 4.1]{BC2016} and \cite[Definition 5.4]{BC2012}}] A multiplier $T$ is called a \textit{full} (resp. \textit{reduced}) multiplier of $\Sigma$ if there exists a bounded linear map on the full (resp. reduced) crossed product associated to $\Sigma$ extending the map $f \mapsto T \cdot f$ (resp. $\Lambda(f) \mapsto \Lambda(T \cdot f)$) on $C_c(\Gamma, A)$ (resp. $\Lambda(C_c(\Gamma, A)))$. 
\end{definition} 

Notice that if $A = \C$, full/reduced multipliers amount to the homonymous concepts for groups, which have attracted much attention, see for example \cite[Chapter 3]{pisier} for a nice introduction.

\begin{definition}
	A multiplier $T$ is called $\Sigma$-\textit{positive definite}, or \textit{positive definite with respect to} $\Sigma$, if for every $n \in \N$ and every choice of $g_1, \ldots, g_n \in \Gamma$ and $a_1, \ldots, a_n \in A$ the matrix $$\left[ \alpha_{g_i} \left( T_{g_i^{-1}g_j} (\alpha_{g_i}^{-1}(a_i^*a_j)) \right) \right]_{1 \leq i, j \leq n}$$ is positive in the matrix algebra $M_n(A)$. The set of $\Sigma$-positive definite multipliers is denoted by $P(\Sigma)$. 
\end{definition}

Note that if $A = \C$, a multiplier of $\Sigma$ is of the form $$T_g(\lambda) = \mu(g) \lambda \quad (g \in \Gamma, \lambda \in \C)$$ for a function $\mu$ on $\Gamma$, in which case $T$ is $\Sigma$-positive definite if and only if $\mu$ is a positive definite function. 

The next example shows the archetypical example of a positive definite multiplier. It also introduces a notation which will be employed frequently in the sequel. 

\begin{example}[{\cite[Example 4.1]{BC2016}}]
	Let $(\rho, v)$ be an equivariant representation of $\Sigma$ on a Hilbert $A$-module $X$ and pick two vectors $\xi, \eta \in X$. One can define a multiplier $T_{\rho, v, \xi, \eta}$ by $$T_{\rho, v, \xi, \eta}(g, a) = \langle \xi, \rho(a) v(g) \eta \rangle \quad (g \in \Gamma, a \in A).$$ It is a simple exercise to use the equivariance relations to show that if $\xi = \eta$, this multiplier becomes $\Sigma$-positive definite. 
\end{example}

The following Gelfand-Raikov type theorem explains how multipliers, equivariant representations and positive definiteness tie together. 

\begin{theorem}[{\cite[Corollary 4.4]{BC2016}}]
	\label{thm:pos_def} 
	For any multiplier $T \in L(\Gamma, A)$, the following conditions are equivalent:  
	\begin{enumerate}
		\item $T$ is $\Sigma$-positive definite. 
		\item There exist an equivariant representation $(\rho, v)$ of $\Sigma$ on a Hilbert $A$-module $X$ and a vector $\xi \in X$ such that $T = T_{\rho, v, \xi, \xi}$. 
		\item $T$ is a reduced multiplier such that the induced map on the reduced crossed product is completely positive. 
		\item $T$ is a full multiplier such that the induced map on the full crossed product is completely positive. 
	\end{enumerate}
\end{theorem}

If $(\rho, v)$ is an equivariant representation of $\Sigma$ on a Hilbert module $X$, a vector $\xi \in X$ is called \textit{cyclic} for $(\rho, v)$ if $$\mathrm{span} \{(\rho(a)v(g) \xi) \cdot b : a, b \in A, g \in \Gamma \}$$ is dense in $X$. It will be useful to know that by \cite[Theorem 4.5]{BC2016}, the vector $\xi$ in condition (ii) above may be chosen to be cyclic. \\

\textbf{Fourier and Fourier-Stieltjes algebras.} The Fourier-Stieltjes algebra of a locally compact group consists of coefficient functions of unitary representations of the group. A GNS-like argument shows that the Fourier-Stieltjes algebra coincides with the span of the positive definite functions on the group. As a Banach space, it can be identified with the dual of the group C*-algebra. The Fourier algebra is the closure inside the Fourier-Stieltjes algebra of the functions with compact support. For more about the classical theory of Fourier and Fourier-Stieltjes algebras, we recommend \cite{kaniuth}. We proceed with a brief recapitulation of the Fourier-Stieltjes algebra of a C*-dynamical system. Proofs and more details can be found in \cite{BC2016}. 

Fix a C*-dynamical system $\Sigma = (A, \Gamma, \alpha)$. 

\begin{definition}
	The \textit{Fourier-Stieltjes algebra} of $\Sigma$ is the span of the $\Sigma$-positive multipliers inside $L(\Gamma, A)$. It is denoted by $B(\Sigma)$. 
\end{definition}

It follows from \cref{thm:pos_def} and polarization that the Fourier-Stieltjes algebra coincides with the set of coefficient functions of equivariant representations and that its elements induce completely bounded maps of the associated crossed products. Moreover, one can show that the Fourier-Stieltjes algebra becomes a Banach algebra with the norm given by $$\norm{T} = \inf \{ \norm{\xi} \norm{\eta} : T = T_{\rho, v, \xi, \eta} \},$$ where the infimum is taken over all equivariant representations implementing $T$ as a coefficient function, see \cite[Proposition 3.1]{BC2016}. The following properties of the Fourier-Stieltjes algebra will be used later in this paper.

\begin{proposition}[{\cite[Proposition 3.2, Corollary 4.3]{BC2016}}]
	\label{prop:FS_alg}
	\begin{enumerate}
		\item For every $T \in B(\Sigma)$, we have $$\sup \{ \norm{T_g} : g \in \Gamma \} \leq \norm{T}.$$ 
		\item There is a contractive embedding $\mu \mapsto T^\mu$ of the Fourier-Stieltjes algebra $B(\Gamma)$ of the group into $B(\Sigma)$, where $T^\mu$ is given by $$T^\mu (g, a) = \mu(g) a \quad (g \in \Gamma, a \in A).$$ 
		\item The set $P(\Sigma)$ of $\Sigma$-positive definite multipliers is a cone inside $B(\Sigma)$. 
	\end{enumerate}
\end{proposition} 

\textbf{Hilbert bundles, Hilbert $C(\Omega)$-modules and their isometries.} Hilbert C*-modules over commutative C*-algebras can be completely described in terms of Hilbert bundles. The reader will find more about Banach bundles in \cite{fell_doran}. For the classification of Hilbert modules over commutative C*-algebras, the reader should consult \cite{dixmier,takahashi_dual,takahashi_rep}. Let $\Omega$ denote a compact Hausdorff space. 

\begin{definition}
	\label{def:banach_bundle}
	A \textit{Banach bundle} over $\Omega$ is a pair $(B, \pi)$ consisting of a topological space $B$ and a continuous open surjection $\pi \colon B \to \Omega$ such that each fiber $B_x := \pi^{-1}(x)$ carries a Banach space structure satisfying the following conditions: 
	\begin{enumerate}
		\item The map $b \mapsto \norm{b}$ from $B$ to $\R$ is continuous. 
		\item The operation $+$ is a continuous function from $\{ (b, c) \in B \times B : \pi(b) = \pi(c) \}$ to $B$. 
		\item For each $\lambda \in \C$, the map $b \mapsto \lambda b$ from $B$ to $B$ is continuous. 
		\item If $x \in \Omega$ and $(b_i)_i$ is net in $B$ such that $\norm{b_i} \xrightarrow{i} 0$ and $\pi(b_i) \xrightarrow{i} x$ in $\Omega$, then $b_i \xrightarrow{i} 0_x,$ where $0_x$ denotes the zero element in $B_x$. 
	\end{enumerate}   
\end{definition}

It follows from the definition that the scalar multiplication map $(\lambda, b) \mapsto \lambda b$ is a continuous map from $\C \times B$ to $B$ and that the subspace topology of $B_x$ inherited from $B$ coincides with the norm topology for each $x \in \Omega$, see \cite[Propositions 13.10 and 13.11]{fell_doran}. A \textit{continuous section} of a Banach bundle $\pi \colon B \to \Omega$ is a continuous map $\xi \colon \Omega \to B$ such that $\pi(\xi(x)) = x$ for every $x \in \Omega$. The space of continuous sections of $\pi \colon B \to \Omega$ is a Banach space when equipped with the norm given by $$\norm{\xi} = \sup \{ \norm{\xi(x)} : x \in \Omega \},$$ which we denote by $\Gamma(\pi)$. 

If $\pi \colon B \to \Omega$ is a Banach bundle such that each fiber is a Hilbert space, $\pi$ is called a \textit{Hilbert bundle}, in which case the definition above guarantees that the inner product is a continuous map from $\{ (b_1, b_2) \in B \times B : \pi(b_1) = \pi(b_2) \}$ to $\C$. If $\pi$ is a Hilbert bundle, its continuous section space $\Gamma(\pi)$ can be organized as a Hilbert $C(\Omega)$-module as follows: The right action of $C(\Omega)$ is given by pointwise scaling, i.e., $$(\xi \cdot f)(x) = f(x) \xi(x) \quad (\xi \in \Gamma(\pi), f \in C(\Omega), x \in \Omega),$$ and the inner product is taken in each fiber, i.e., $$\langle \xi, \eta \rangle (x) = \langle \xi(x), \eta(x) \rangle \quad (\xi, \eta \in \Gamma(\pi), x \in \Omega).$$ In fact, every Hilbert $C(\Omega)$-module arises this way, see for example \cite[Corollary 3.13]{takahashi_rep}.

\begin{theorem}
	\label{thm:hilb_mods}
	Every Hilbert $C(\Omega)$-module is unitarily equivalent to the continuous section space of a Hilbert bundle over $\Omega$.  
\end{theorem}

\begin{example}
	\label{ex:triv_bundle}
	Fix a Hilbert space $H$ and consider the trivial bundle $B = \Omega \times H$, where the map $\pi \colon B \to \Omega$ is the projection onto the first coordinate. The continuous section space of this bundle is $C(\Omega, H)$, the space of continuous $H$-valued maps on $\Omega$ with the supremum norm. 
\end{example}

The classical Banach-Stone theorem provides a useful description of the surjective linear isometries of spaces of complex-valued functions on a compact Hausdorff space. There exist generalizations of this result to Banach space valued functions under different hypotheses. For the purview of this paper, the article \cite{HsuWong2011} by Hsu and Wong on Banach-Stone like theorems for continuous section spaces of Banach bundles is particularly interesting. For the case of Hilbert bundles, their theorem can be stated as follows: 

\begin{theorem}[{\cite[Theorem 1.3]{HsuWong2011}}]
	\label{thm:hsu_wong}
	Let $\pi \colon H \to \Omega$ be a Hilbert bundle over $\Omega$ with fibers $H_x := \pi^{-1}(x)$ ($x \in \Omega$), and suppose that $V \colon \Gamma(\pi) \to \Gamma(\pi)$ is a surjective isometry of the continuous section space of $\pi$. Then there exist a homeomorphism $\sigma \colon \Omega \to \Omega$ and a unitary operator $u(x) \colon H_{\sigma(x)} \to H_x$ for each $x \in \Omega$ such that $$(V \xi)(x) = u(x) \xi(\sigma(x)) \quad (\xi \in \Gamma(\pi), x \in \Omega).$$ 
\end{theorem}

For the trivial bundle described in \cref{ex:triv_bundle}, this result goes at least as far back as \cite[Theorem 6.2]{jerison}. 

\section{The Fourier algebra} 
\label{chap_fourier}

In their article \cite[Remark 4.9]{BC2016}, Bédos and Conti briefly discuss some possible definitions for a Fourier algebra of a C*-dynamical system. In \cite{skalski}, Buss et al. give a definition based on regular equivariant representations. We suggest a definition which at least formally differs from the one suggested in \cite{skalski}. However, we are able to show that the Fourier algebra coincides with the closure of the multipliers that are coefficients of regular equivariant representations.  

Fix a C*-dynamical system $\Sigma = (A, \Gamma, \alpha)$. The \textit{support} of a multiplier $T \in L(\Gamma, A)$ is the set $\{g \in \Gamma : T_g \neq 0 \}$. 

\begin{definition}
	The \textit{Fourier algebra} $F(\Sigma)$ of $\Sigma$ is the closure inside $B(\Sigma)$ of the multipliers with finite support. 
\end{definition}

Since the property of having finite support is preserved by composition from both the left and right, it is clear that $F(\Sigma)$ sits as a two-sided ideal inside the Fourier-Stieltjes algebra. One immediately wonders if the Fourier algebra can be described in terms of regular equivariant representations. As seen in \cref{ex:reg_equiv}, there is a plethora of such representations. The following equivariant Fell absorption principle suggests that all regular equivariant representation are necessary in order to recreate the ideal $F(\Sigma)$, as we will indeed see in \cref{thm:fourier_alg}. The following result follows from \cite[Proposition 6.13]{skalski}, but we include the proof in our setting for the convenience of the reader. 

\begin{proposition}
	\label{prop:equiv_fell}
	Let $(\rho, v)$ be an equivariant representation of $\Sigma$ on a Hilbert $A$-module $X$. The tensor product equivariant representation $(\rho \otimes \check{\ell}, v \otimes \check{\alpha})$ is unitarily equivalent to the regular equivariant representation $(\check{\rho}, \check{v})$ associated to $(\rho, v)$.
\end{proposition}

\begin{proof} The proof follows that of \cite[Theorem 4.11]{BC2012} quite closely. Firstly, there exists a unitary map $W \colon X \otimes_{\check{\ell}} A^\Gamma \to X^\Gamma$ such that $$ \left( W(x \dotimes \xi)  \right)(g) = x \cdot \xi(g) $$ for every $x \in X, \xi \in C_c(\Gamma, A)$ and $g \in \Gamma$. Indeed, let $W$ be given by this prescription on the dense subspace of $X \otimes_{\check{\ell}} A^\Gamma$ spanned by elements of the form $x \dotimes \xi$ for $x \in X$ and $\xi \in C_c(\Gamma, A)$. The computation 
	\begin{align*}
		\langle W(x \dotimes \xi), W(y \dotimes \eta) \rangle &= \sum_{g \in \Gamma} \langle x \cdot \xi(g), y \cdot \eta(g) \rangle \\
		&= \sum_{g \in \Gamma} \xi(g)^* \langle x, y \rangle \eta(g) \\
		&= \langle \xi , \check{\ell}( \langle x, y \rangle) \eta \rangle \\
		&= \langle x \dotimes \xi, y \dotimes \eta \rangle 
	\end{align*}
	holds for all $x, y \in X$ and $\xi, \eta \in C_c(\Gamma, A)$, from which it follows that $W$ extends to an isometry of $X \otimes_{\check{\ell}} A^\Gamma$ into $X^\Gamma$. The image of $W$ is easily seen to be dense, and so $W$ is surjective. Moreover, $W$ is readily checked to be $A$-linear. In conclusion, $W$ is a unitary map. Furthermore, in the same notation, we have
	\begin{align*}
		\check{\rho}(a)W(x \dotimes \xi)(h) &= \rho(a) \left(W(x \dotimes \xi) \right)(h) \\
		&= \rho(a)(x \cdot \xi(h)) = \rho(a)x \cdot \xi(h) \\
		&= W(\rho \otimes \check{\ell}(a))(x \dotimes \xi) (h), \\
		\check{v}(g) W (x \dotimes \xi)(h) &= v(g) W(x \dotimes \xi)(g^{-1}h) \\
		&= v(g) (x \cdot \xi(g^{-1}h)) \\
		&= (v(g)x) \cdot \alpha_g( \xi(g^{-1}h)) \\
		&= W ( v(g)x \dotimes \check{\alpha}(g) \xi) (h) \\
		&= W (v \otimes \check{\alpha})(g) (x \dotimes \xi)(h),
	\end{align*}
	which finishes the proof. 
\end{proof}

In \cite{skalski}, Buss et al. define the set of \textit{Fourier multipliers} as the set of multipliers which can be implemented as a coefficient function of a regular equivariant representation, see \cite[Definition 9.1]{skalski}. They show that with this convention, every multiplier with compact support is a Fourier multiplier, However, the converse is left as an open question, see \cite[Question 12.3]{skalski}. Our next result shows that with our definition, the corresponding question can be answered in the affirmative. Note that it is shown that any coefficient function of a regular equivariant representation is a limit of multipliers with finite support in the norm in $B(\Sigma)$, which provides a partial answer to \cite[Question 12.3]{skalski} in the case of a discrete unital C*-dynamical system. 

\begin{theorem}
	\label{thm:fourier_alg}
	Let $(\rho, v)$ be an equivariant representation of $\Sigma$ on a Hilbert $A$-module $X$. Consider the associated regular representation $(\check{\rho}, \check{v})$ of $\Sigma$ on $X^\Gamma$. Pick $\xi, \eta \in X^\Gamma$ and let $T = T_{\check{\rho}, \check{v}, \xi, \eta}$ be the associated multiplier. Then there exists a net $(T_i)_{i \in I}$ of multipliers in $F(\Sigma)$ with finite support such that $\norm{T_i} \leq \norm{T}$ for every $i \in I$ converging to $T$. Conversely, if $T \in F(\Sigma)$ has finite support, $T$ is a coefficient function of a regular equivariant representation of $\Sigma$. Succinctly, $F(\Sigma)$ is the closure in $B(\Sigma)$ of coefficient functions of regular equivariant representations. 
\end{theorem}
\begin{proof}
	We start with the first statement. Notice that it suffices to prove the statement for coefficient functions of the regular equivariant representation. Indeed, since the multiplication in the Fourier-Stieltjes algebra is implemented via tensoring on the representation level, see \cite[Proof of Lemma 3.1]{BC2016}, it follows from \cref{prop:equiv_fell} that any coefficient function of a regular equivariant representation factors as the product of one multiplier with a coefficient of the regular equivariant representation. Hence, approximating the factor coming from the regular equivariant representation gives an approximation of the product with the desired property since having finite support is preserved and the multiplication is continuous. 
	
	For a vector $\xi \in A^\Gamma$ and a subset $S \subseteq \Gamma$, we denote by $\xi_S \colon \Gamma \to A$ the function that agrees with $\xi$ on $S$ and is zero elsewhere. Clearly, if $\xi \in A^\Gamma$ and $S$ is finite, both $\xi_S$ and $\xi_{S^c} = \xi - \xi_S $ are in $A^\Gamma$. Moreover, it is easy to check that if $\xi, \eta \in A^\Gamma$ and $S_1, S_2 \subseteq \Gamma$ are finite, the multiplier $T_{\check{\ell}, \check{\alpha}, \xi_{S_1}, \eta_{S_2}}$ associated to $\xi_{S_1}$ and $\eta_{S_2}$ has finite support. 
	
	Fix non-zero $\xi, \eta \in A^\Gamma$ and consider the associated multiplier $T := T_{\check{\ell}, \check{\alpha}, \xi, \eta}$. We will show that $T$ is a limit inside $B(\Sigma)$ of multipliers with finite support. Indeed, fix $\varepsilon > 0$. The fact that $\xi \in A^\Gamma$ means that the series $$\sum_{g \in \Gamma} \xi(g)^* \xi(g)$$ converges (unconditionally) in $A$. Thus, we may pick a finite subset $S_1 \subseteq \Gamma$ such that $$\norm{ \sum_{g \in \Gamma} \xi(g)^* \xi(g) - \sum_{g \in S_1} \xi(g)^* \xi(g) } = \norm{ \sum_{g \in S_1^c} \xi(g)^* \xi(g) } < \frac{\varepsilon}{2\norm{\eta}}.$$ Analogously, there is a finite subset $S_2 \subseteq \Gamma$ such that $$\norm{ \sum_{g \in \Gamma } \eta(g)^* \eta(g) - \sum_{g \in S_2} \eta(g)^* \eta(g) } = \norm{\sum_{g \in S_2^c} \eta(g)^* \eta(g) } < \frac{\varepsilon}{2 \norm{\xi_{S_1}}}.$$ Set $T_{\varepsilon} := T_{\check{\ell}, \check{\alpha}, \xi_{S_1}, \eta_{S_2}}$, which has finite support by the comments in the beginning of the proof. A straightforward computation shows that $$T_{\check{\ell}, \check{\alpha}, \xi, \eta} - T_{\check{\ell}, \check{\alpha}, \xi_S, \eta} = T_{\check{\ell}, \check{\alpha}, \xi_{S^c}, \eta} \quad \text{and} \quad T_{\check{\ell}, \check{\alpha}, \xi, \eta} - T_{\check{\ell}, \check{\alpha}, \xi, \eta_S} = T_{\check{\ell}, \check{\alpha}, \xi, \eta_{S^c}}$$ for any finite subset $S \subseteq \Gamma$. Combining these observations, we see that
	\begin{align*}
		\norm{T - T_\varepsilon} &= \norm{(T_{\check{\ell}, \check{\alpha}, \xi, \eta} - T_{\check{\ell}, \check{\alpha}, \xi_{S_1}, \eta}) + (T_{\check{\ell}, \check{\alpha}, \xi_{S_1}, \eta} - T_{\check{\ell}, \check{\alpha}, \xi_{S_1}, \eta_{S_2}})} \\
		&= \norm{T_{\check{\ell}, \check{\alpha}, \xi_{S_1^c}, \eta} + T_{\check{\ell}, \check{\alpha}, \xi_{S_1}, \eta_{S_2^c}}} \\
		& \leq \norm{T_{\check{\ell}, \check{\alpha}, \xi_{S_1^c}, \eta} } + \norm{ T_{\check{\ell}, \check{\alpha}, \xi_{S_1}, \eta_{S_2^c}}} \\
		& \leq \norm{\xi_{S_1^c}} \norm{\eta} + \norm{\xi_{S_1}} \norm{\eta_{S_2^c}} \\
		& < \frac{\varepsilon}{2} + \frac{\varepsilon}{2}= \varepsilon, 
	\end{align*}
	which shows that $T$ is a limit of multipliers in $B(\Sigma)$ with finite support, so $T \in F(\Sigma)$. Moreover, observe that $$\norm{T_\varepsilon} \leq \norm{\xi_{S_1}} \norm{\xi_{S_2}} \leq \norm{\xi} \norm{\eta},$$ where the last inequality follows from the monotonicity of the norm on positive elements in a C*-algebra. For any $\delta > 0$, we may assume that $\xi, \eta$ are chosen so that $\norm{\xi} \norm{\eta} < \norm{T} + \delta.$ Hence, we have shown that $\norm{T_\varepsilon} \leq \norm{T} + \delta$ for any $\delta > 0$, and we conclude that $\norm{T_\varepsilon} \leq \norm{T}$, which finishes the proof of the first part of the theorem. 
	
	For the converse, suppose that $T \in F(\Sigma)$ has finite non-empty support $S$. Let $(\rho, v)$ be an equivariant representation of $\Sigma$ on a Hilbert $A$-module $X$ implementing $T$ via vectors $\xi, \eta \in X$. Let $(\check{\rho}, \check{v})$ denote the associated regular representation. Let $\xi \odot \delta_S$ denote the element of $C_c(\Gamma, X)$ which is $\xi$ on $S$ and zero elsewhere. Similarly, let $\eta \odot \delta_e$ denote the function which is $\eta$ at $e \in \Gamma$ and zero elsewhere. The coefficient of $(\check{\rho}, \check{v})$ associated to $\xi \odot \delta_S, \eta \odot \delta_e \in X^\Gamma$ is given by 
	\begin{align*}
		\langle \xi \odot \delta_S, \check{\rho}(a) \check{v}(g) (\eta \odot \delta_e) \rangle &= \sum_{h \in S} \langle \xi, \rho(a) v(g) (\eta \odot \delta_e)(g^{-1}h) \rangle \\
		&= \begin{cases}
			\langle \xi, \rho(a) v(g) \eta \rangle \quad & \text{if } g \in S, \\
			0 \quad & \text{otherwise}, \\
		\end{cases}  \\
		&= T_g(a)
	\end{align*} 
	for every $a \in A$ and $g \in \Gamma$. Hence, $T$ is a coefficient of a regular equivariant representation, which shows the second assertion. It follows that every element of $F(\Sigma)$ is a limit of coefficient functions of regular equivariant representations. 
\end{proof}

Combining the last part of the previous proof with polarization yields the following corollary, which shows that our definition of the Fourier algebra agrees with one of the ones suggested in \cite[Remark 4.9]{BC2016}.

\begin{corollary}
	The Fourier algebra $F(\Sigma)$ coincides with the closed linear span of the $\Sigma$-positive definite multipliers with finite support. 
\end{corollary}

\begin{proof} Clearly, the closed linear span of $\Sigma$-positive definite multipliers with finite support is contained in the Fourier algebra. To show the opposite inclusion, consider an element $T \in F(\Sigma)$ and let $\varepsilon > 0$. We may pick a multiplier $T'$ with finite support such that $\norm{T - T'} < \varepsilon$. By the last part of the previous proof, $T' = T_{\check{\rho}, \check{v}, \xi, \eta}$ for some equivariant representation $(\rho, v)$ of $\Sigma$ on a Hilbert $A$-module $X$ and some $\xi, \eta \in X^\Gamma$ with finite support. By polarization, we have $$T' = \frac{1}{4} \sum_{k=0}^3 i^k T_{\check{\rho}, \check{v}, \xi_k, \xi_k}, $$ where $\xi_k = i^k \xi + \eta$ for $k \in \{0, 1, 2, 3\}.$ Since $\xi$ and $\eta$ have finite support, so does $\xi_k$ for each $k \in \{0, 1, 2, 3 \}.$ Hence, $T'$ is a linear combination of $\Sigma$-positive definite multipliers with finite support. Since $\varepsilon > 0$ was arbitrary, this finishes the proof. 
\end{proof} 

As already mentioned, Buss et al. define their analogue of the Fourier algebra as the set of coefficient functions on regular equivariant representations. Clearly, this is a subset of $F(\Sigma)$. The theorem above should be compared to \cite[Proposition 9.7]{skalski}. It remains open if the set of coefficient functions of regular equivariant representations is closed in $F(\Sigma)$, and so if the two definitions coincide, see \cite[Question 12.1]{skalski}. In \cite{skalski}, Buss et al. equip the set of coefficient functions of regular equivariant representations with its own norm by taking the infimum of products of norms of vectors implementing the multiplier via a regular equivariant representation. The set of coefficient functions of such representations is then complete with respect to this norm, for essentially the same reason that the Fourier-Stieltjes algebra is complete since direct sums of regular equivariant representations are again regular. The authors point to \cite[Definition 1.4]{renault} and \cite{oty}  and concede that it has been the convention to take the closure in the Fourier-Stieltjes norm when extending the Fourier algebra to groupoid settings. We believe that \cref{thm:fourier_alg} lends merit to our convention. We close this section by listing some properties of the Fourier algebra. 

\begin{proposition} 
	Let $\Sigma$ denote a C*-dynamical system. 
	\begin{enumerate}
		\item The Fourier algebra $F(\Sigma)$ is commutative if and only if the Fourier-Stieltjes algebra $B(\Sigma)$ is. 
		\item The embedding of $B(\Gamma)$ into $B(\Sigma)$ mentioned in \cref{prop:FS_alg} restricts to a contractive embedding of $F(\Gamma)$ into $F(\Sigma)$. 
		\item Every element $T \in F(\Sigma)$ vanishes at infinity in the sense that for every $\varepsilon > 0$ the set $\{g \in \Gamma : \norm{T_g} \geq \varepsilon \}$ is finite. 
		\item The Fourier algebra is unital if and only if the group is finite. 
	\end{enumerate}
\end{proposition}

\begin{proof}
	Clearly, $F(\Sigma)$ is commutative if $B(\Sigma)$ is. Conversely, pick $T, S \in B(\Sigma), g \in \Gamma$ and $a \in A$ such that $(S T)(g,a)  \neq (T S)(g, a)$. For $i \in \{1, 2\}$ let $(\rho_i, v_i)$ be an equivariant representation of $\Sigma$ on a Hilbert $A$-module $X_i$ implementing $T$ and $S$ respectively via vectors $x_i, y_i \in X_i$. Let $(\check{\rho_i}, \check{v_i})$ denote the corresponding regular equivariant representation on $X_i^\Gamma$ for $i \in \{1, 2 \}$. Consider the vectors $x_i \odot \delta_g, y_i \odot \delta_e \in X^\Gamma$, following the notation from the proof of \cref{thm:fourier_alg}. It is easy to check that the multiplier $\check{T}$ associated to $(\check{\rho_1}, \check{v_1})$ and $x_1 \odot \delta_g, y_1 \odot \delta_e \in X^\Gamma$ is given by $$\check{T}(h, a) = \begin{cases}
		T(g, a), \quad & \text{if } h = g, \\
		0, \quad & \text{otherwise}. 
	\end{cases}$$ \cref{thm:fourier_alg} shows that $\check{T} \in F(\Sigma)$. By defining $\check{S}$ similarly, we get two elements $\check{T}, \check{S} \in F(\Sigma)$ that do not commute, which shows (i). For (ii), note that if $\mu \in B(\Gamma)$ has finite support, so does $T^\mu \in B(\Sigma)$. Hence, by continuity of the embedding $B(\Gamma) \to B(\Sigma)$, the image of $F(\Gamma)$ is in $F(\Sigma)$. Towards part (iii), consider $T \in F(\Sigma)$ and suppose that $\varepsilon > 0$ is given. Pick $T' \in B(\Sigma)$ with finite support such that $\norm{T - T'} < \varepsilon$. By \cref{prop:FS_alg} (i), $\norm{T_g - T_g'} \leq \varepsilon$ for every $g \in \Gamma$ and it follows from the triangle inequality that $\norm{T_g} \leq \varepsilon$ except possibly on the support of $T'$, which is finite. For the last statement, observe first that it follows from the proof of (i) that if $E$ is a unit in $F(\Sigma)$, it is a unit in $B(\Sigma)$. However, by (iii) it must vanish at infinity, which is only possible if $\Gamma$ is finite. 
\end{proof}

\section{Equivariant representations of commutative systems} 
\label{chap_comm}

Towards an understanding of the rigidity properties of the Fourier-Stieltjes algebra, it may be helpful to analyze some examples. In this chapter, we will first give a description of the group theoretical aspect of the equivariant representation theory of general commutative systems. In fact, it will be shown that the group representation in an equivariant representation of a commutative C*-dynamical system can be described in terms of a cocycle representation of the corresponding action on the Gelfand spectrum. This concept seems to only have been studied for measure spaces in the literature, see \cite[Section 2]{zimmer}, but the definitions have obvious generalizations to topological spaces. Finally, we compute the Fourier-Stieltjes algebra of two classes of systems where finite cyclic groups act on finite-dimensional commutative C*-algebras. 

Let $\Gamma$ be a discrete group acting by homeomorphisms on a compact Hausdorff space $\Omega$. Denote the corresponding C*-dynamical system by $\Sigma = (C(\Omega), \Gamma, \alpha)$. 
\begin{definition}
	\label{def:cocycle_rep}
	Let $\pi \colon H \to \Omega$ be a Hilbert bundle over $\Omega$ with fibres $H_x := \pi^{-1}(x)$ for $x \in \Omega$. A \textit{cocycle representation} of the transformation group $(\Gamma, \Omega)$ on $\pi$ consists of a unitary operator $u(x, g) \colon H_{g^{-1}x} \to H_x$ for each $(x, g) \in \Omega \times \Gamma$ such that 
	\begin{enumerate}
		\item given two continuous sections $\xi \colon \Omega \to H, \eta \colon \Omega \to K$ of $\pi_1$ and $\pi_2$ respectively, the function $$(x, g) \mapsto \langle u(x, g) \xi(g^{-1}x), \eta(x) \rangle$$ is a continuous map on $\Omega \times \Gamma$. 
		\item the cocycle identity $$u(x, gh) = u(x, g) u(g^{-1}x, h)$$ holds for all $g, h \in \Gamma$ and $x \in \Omega$. 
	\end{enumerate}
	
	Two cocycle representations $u_1$ and $u_2$ on Hilbert bundles $\pi_1 \colon H \to \Omega$ and $\pi_2 \colon K \to \Omega$ respectively are called \textit{equivalent} if for every $x \in \Omega$ there exists a unitary map $U(x) \colon H_x \to K_x$ such that 
	\begin{enumerate}
		\item the equality $$U(x) u_1(x, g) U(g^{-1}x)^* = u_2(x, g)$$ holds for all $(x, g) \in \Omega \times \Gamma$. 
		\item for every pair of continuous sections $\xi, \eta \colon \Omega \to H$ of $\pi_1$, the function $$x \mapsto \langle U(x) \xi(x), \eta(x) \rangle$$ is a continuous map on $\Omega$. 
	\end{enumerate}
\end{definition}

\begin{theorem}
	Let $\pi \colon H \to \Omega$ be a Hilbert bundle over $\Omega$ with fibres $H_x := \pi^{-1}(x)$ for $x \in \Omega$. Recalling that $\Gamma(\pi)$ is a Hilbert $C(\Omega)$-module, suppose that $v \colon \Gamma \to \I(\Gamma(\pi))$ is a group homomorphism satisfying conditions (ii) and (iii) in \cref{def:equiv}. Then there exists a cocycle representation $(x, g) \mapsto u(x, g)$ of $(\Gamma, \Omega)$ on $\pi$ such that \begin{equation} 
		\label{eq:cocycle_equiv} 
		\left( v(g) \xi \right)(x) = u(x, g) \xi(g^{-1}x) \quad (g \in \Gamma, \xi \in \Gamma(\pi), x \in \Omega).
	\end{equation} 
	Conversely, given a cocycle representation $u$ of $(\Gamma, \Omega)$ on a Hilbert bundle $\pi \colon H \to \Omega$, \eqref{eq:cocycle_equiv} defines a group homomorphism $v$ from $\Gamma$ into $\I(\Gamma(\pi))$ satisfying conditions (ii) and (iii) in \cref{def:equiv}. At last, two such group homomorphisms agree up to unitary equivalence if and only if the underlying cocycle representations are unitarily equivalent. 
\end{theorem}

\begin{proof}
	By \cref{thm:hsu_wong}, there exists for each $g \in \Gamma$ a homeomorphism $\sigma_g \colon \Omega \to \Omega$ and a unitary map $u(x, g) \colon H_{\sigma_g(x)} \to H_x$ for each $x \in \Omega$ such that $$\left( v(g) \xi \right) (x) = u(x, g) \xi (\sigma_g(x)) \quad (g \in \Gamma, \xi \in \Gamma(\pi), x \in \Omega).$$ We first show that the homeomorphisms agree with the action of $\Gamma$, i.e., $\sigma_g(x) = g^{-1}x$ for $g \in \Gamma$ and $x \in \Omega$ . Indeed, fix also $g \in \Gamma$ and $x \in \Omega$. Using \cref{def:equiv} (iii), we get that $$f(\sigma_g(x)) u(x, g) \xi(\sigma_g(x)) = f(g^{-1}x) u(x, g) \xi(\sigma_g(x))$$ for every $f \in C(\Omega)$ and $\xi \in \Gamma(\pi)$. Hence, $f(\sigma_g(x)) = f(g^{-1}x)$ for every $f \in C(\Omega)$, whence it follows that $\sigma_g(x) = g^{-1}x$ by Urysohn's lemma. Furthermore, since $v$ is a group homomorphism, we have for $g, h \in \Gamma, \xi \in \Gamma(\pi)$ and $x \in \Omega$ that
	\begin{align*}
		\left( v(gh) \xi \right)(x) &= u(x, gh) \xi( (gh)^{-1}x) \\
		&= \left( v(g) v(h) \xi \right)(x) \\
		&= u(x, g) u(g^{-1}x, h) \xi ((gh)^{-1}x).
	\end{align*}
	Since $\Omega$ is compact, there is a continuous section passing through every vector in $H$, see for example \cite[Remark 13.19]{fell_doran}. We conclude that $$u(x, gh) = u(g, x) u(g^{-1}x, h)$$ holds for all $g, h \in \Gamma$ and every $x \in \Omega$. At last, consider two continuous sections $\xi, \eta \colon \Omega \to H$ of $\pi$. Since the inner product on a Hilbert bundle is a continuous map on $\{ (\xi, \eta) \in H \times H : \pi(\xi) = \pi(\eta) \}$ to $\C$, it suffices to check that the map
	\begin{align*}
		\Omega \times \Gamma & \to H \\
		(x, g) & \mapsto u(x, g) \xi(g^{-1}x) 
	\end{align*}
	is continuous in order to conclude that (i) in \cref{def:cocycle_rep} holds. Moreover, since $\Gamma$ is discrete, it suffices to argue that this map is continuous in $x$ for each fixed $g \in \Gamma$. To that end, consider a net $(x_i)_i$ in $\Omega$ converging to some $x \in \Omega$. Then $$\norm{u(x, g) \xi(g^{-1}x) - u(x_i, g) \xi(g^{-1}x_i)} = \norm{(v(g)\xi)(x) - (v(g)\xi)(x_i)} \xrightarrow{i} 0$$ since $v(g)$ maps $\Gamma(\pi)$ into itself. For the second statement, it suffices to invert all the computations and arguments we have given up to this point. Towards the last statement, suppose that $u_1$ and $u_2$ are unitarily equivalent cocycle representations and suppose that $U(x)$ satisfies (i) and (ii) in the second part of \cref{def:cocycle_rep} for each $x \in \Omega$. Define $U \colon \Gamma(\pi) \to \Gamma(\pi)$ by $(U \xi)(x) = U(x) \xi(x)$ for each $\xi \in \Gamma(\pi)$ and $x \in \Omega$. It is easy to check that $U$ intertwines the group homomorphisms induced by $u_1$ and $u_2$. 
\end{proof}

To give a concrete description of the possibilities for the representation $\rho$ of $C(\Omega)$ in an equivariant representation $(\rho, v)$ of $(C(\Omega), \Gamma, \alpha)$ remains more elusive. Classification results for left actions of commutative C*-algebras on their Hilbert modules have been discussed in the literature under certain finiteness and compatibility assumptions, see for example \cite{adamo}. However, we are able to produce a class of equivariant representations inspired by \cite[Example 3.4]{adamo}. Notably, we are able to identify an algebra representation to pair up with every group representation. 

Recall that a continuous map $\sigma \colon \Omega \to \Omega$ is called $\Gamma$\textit{-equivariant} if $\sigma(g x) = g \sigma(x)$ for every $g \in \Gamma$ and $x \in \Omega$, i.e., $\sigma$ commutes with the action of $\Gamma$. An obvious example is the identity map on $\Omega$. 

\begin{theorem}
	\label{thm:comm_rep} 
	Let $\pi \colon H \to \Omega$ be a Hilbert bundle over $\Omega$, and suppose that $u$ is a cocycle representation of $(\Gamma, \Omega)$ on $\pi$ with associated group homomorphism $v$. For any $\Gamma$-equivariant map $\sigma \colon \Omega \to \Omega$, define $\rho \colon C(\Omega) \to \L(\Gamma(\pi))$ by $$\left( \rho(f) \xi \right) (x) = f(\sigma(x)) \xi(x) \quad ( f \in C(\Omega), \xi \in \Gamma(\pi), x \in \Omega).$$ Then $(\rho, v)$ is an equivariant representation of $\Sigma$ on $\Gamma(\pi)$. 
\end{theorem}

\begin{proof}
	It is clear that $\rho$ is a representation of $C(\Omega)$ on $\Gamma(\pi)$ by adjointable operators. Hence, it suffices to check that $(\rho, v)$ satisfies the first equivariance relation. To that end, pick $f \in C(\Omega), g \in \Gamma$ and $\xi \in \Gamma(\pi)$. Using $\Gamma$-equivariance of $\sigma$, we have for every $x \in \Omega$ that  
	$$		\left( \rho(\alpha_g(f)) v(g) \xi \right)(x) = f(g^{-1} \sigma(x)) u(x,g) \xi(g^{-1}x), $$
	and
	\begin{align*}
		\left( v(g) \rho(f) \xi \right)(x) &= u(x, g) \left( f(\sigma(g^{-1}x)) \xi(g^{-1}x) \right) \\
		&= f(g^{-1} \sigma(x)) u(x, g) \xi(g^{-1}x), 
	\end{align*}
	which is what we needed to see. 
\end{proof}

As mentioned in \cref{chap_prelim}, every equivariant representation of $(C(\Omega), \Gamma, \alpha)$ gives rise to a C*-correspondence over $C^*(C(\Omega), \Gamma, \alpha)$ and $C^*_{\mathrm{r}} (C(\Omega), \Gamma, \alpha)$. Hence, the theorem above may have applications to constructions of such. \\

\textbf{Examples.} We close this section by analyzing two classes of examples where finite cyclic groups act on finite-dimensional commutative C*-algebras. Though these systems might be considered as toy examples, they do shed some light on rigidity properties of the Fourier-Stieltjes algebra. Importantly, we identify two classes of examples of systems indexed by $n \in \N$ that are not cocycle group conjugate, see \cite[Definition 3.6]{BC2021}, but have isomorphic Fourier-Stieltjes algebras. It is also worth noting that in these examples, the equivariant representations described in \cref{thm:comm_rep} suffice to generate the Fourier-Stieltjes algebras. 

The cyclic group of order $n \in \N$ is denoted by $\Z_n = \{0, 1, \ldots, n-1 \}$. The standard basis in $\C^n$ is denoted by $e_k$ for $k \in \Z_n$ and the standard matrix units in $M_n(\C)$ are denoted by $E_{kl}$ for $k, l \in \Z_n$. 

\begin{example}
	For each $n \in \N$, let $\Omega_n = (\C^n, \Z_n, \mathrm{id})$ denote the C*-dynamical system where $\Z_n$ acts trivially on $\C^n$. Fix $k, l, p \in \Z_n$. Let $H = 0 \sqcup \cdots \sqcup \C^n \sqcup \cdots \sqcup 0$ be the disjoint union of $n-1$ trivial Hilbert spaces and $\C^n$ at the $k$th place. Formally, $H = \bigcup_{m \in \Z_n} (X_m \times \{m \})$, where $X_m = 0$ for $m \neq k$ and $X_k = \C^n$. Let $\pi \colon H \to \Z_n$ be the map $(\xi, m) \mapsto m$. The section space of this bundle can be identified with $\C^n$ with action by $\C^n$ and inner product given by 
	\begin{align*} 
		(x \cdot a)_j &= x_j a_k \\
		\langle x, y \rangle  &= \left( \sum_{j \in \Z_n} \overline{x_j} y_j \right) e_k
	\end{align*} 
	for $a, x, y \in \C^n$ and $j \in \Z_n$. Let $\sigma_l \colon \Z_n \to \Z_n$ be constantly equal to $l$ and let $U(x, 1)$ be the shift matrix $$U(x, 1) = \begin{pmatrix}
		0 & 0 & \cdots & 0 & 1 \\
		1 & 0 & \cdots  & 0 & 0 \\
		\vdots & \vdots & \ddots & \vdots & \vdots \\
		0 & 0 & \cdots & 1 & 0 
	\end{pmatrix}$$ for each $x \in \Z_n$. By defining $u(x, m) = U(x, 1)^m$ for $(x, m) \in \Z_n \times \Z_m$ we get a cocycle representation on $\pi$. Under the identification above, the representation $\rho_l$ of $\C^n$ associated to $\sigma_l$ is given by $$(\rho_l (a) \xi)_m = \xi_m a_l \quad (a, \xi \in \C^n, m \in \Z_n)$$ and the group homomorphism $v$ associated to $u$ is given by $$v(1) (\xi_0,\xi_1, \ldots, \xi_{n-1}) = (\xi_{n-1}, \xi_0, \ldots, \xi_{n-2}) \quad ((\xi_0, \ldots, \xi_{n-1}) \in \C^n)$$ and $v(m) = v(1)^m$ for $m \in \Z_n$. A simple computation shows that the multiplier associated to the equivariant representation $(\rho_l, v)$ and the choice of vectors $e_p, e_0 \in \C^n$ is given by \begin{equation} 
		\label{eq:multiplier} 
		T(m, a) = \delta_{p, m} E_{kl} a \quad (m \in \Z_n, a \in \C^n).
	\end{equation} Hence, we have implemented the standard matrix units in $L(\Z_n, \C^n) \cong \bigoplus_{k \in \Z_n} M_n(\C)$ as coefficient functions of equivariant representations of $\Omega_n$. We conclude that $$B(\Omega_n) \cong \bigoplus_{k \in \Z_n} M_n(\C).$$
\end{example}

\begin{example}
	The cyclic group $\Z_n$ also has a natural action $\alpha$ on $\C^n$ by cyclically permuting the indices, i.e., $$\alpha_1(\lambda_0, \lambda_1, \ldots, \lambda_{n-1}) = (\lambda_{n-1}, \lambda_0, \ldots, \lambda_{n-2}) \quad ((\lambda_0, \lambda_1, \ldots, \lambda_{n-1}) \in \C^n),$$ and $\alpha_k = (\alpha_1)^k$ for $k \in \Z_n$. Denote the corresponding C*-dynamical system by $\Sigma_n$. Now, let $H = \C^n \sqcup \cdots \sqcup \C^n$ denote the disjoint union of $n$ copies of $\C^n$. As above, $H = \bigcup_{m \in \Z_n} (\C^n \times \{m \})$ as a set. Let $\pi \colon H \to \Z_n$ be the map $(\xi, m) \mapsto m$. The section space of this bundle can be identified with the Hilbert $\C^n$-module $X$ whose underlying vector space is $\bigoplus_{k \in \Z_n} \C^n$, where $\C^n$ acts according to $$(\xi_0, \ldots, \xi_{n-1}) \cdot (\lambda_0, \ldots, \lambda_{n-1}) = (\lambda_0 \xi_0, \ldots, \lambda_{n-1} \xi_{n-1}),$$ and the inner product is taken in each summand. Let $\sigma \colon \Z_n \to \Z_n$ be the identity function, and let $U(x, 1)$ be given as in the previous example for each $x \in \Z_n$. By again setting $u(x, m) = U(x,1)^m$, we get a cocycle representation on $\pi$. The equivariant representation $(\rho, v)$ induced by $\sigma$ and $u$ is given by 
	\begin{align*}
		\rho(a)(\xi_0, \ldots, \xi_{n-1}) &= (a \xi_0, \ldots, a \xi_{n-1}) \\
		v(1)(\xi_0, \ldots, \xi_{n-1}) &= (\alpha_1(\xi_{n-1}), \ldots, \alpha_1(\xi_{n-2})) \\
		v(m) &= v(1)^m 
	\end{align*} 
	for $(\xi_0, \ldots, \xi_{n-1}) \in \bigoplus_{m \in \Z_n} \C^n$ and $m \in \Z_n$. Let $x \in X$ be the vector with $e_l$ in its $k$th component and zero elsewhere, and let $y \in X$ be the vector with $e_{p+l}$ in every component. A straightforward computation shows that the multiplier associated to the equivariant representation defined by $\sigma$ and $U$ and the vectors $x, y \in X$ is given by \cref{eq:multiplier}. Again, we conclude that the Fourier-Stieltjes algebra of $\Sigma_n$ is isomorphic to $$\bigoplus_{k \in \Z_n} M_n(\C),$$ which is somewhat more surprising than in the previous example.  
\end{example}

Combined, the two previous examples show that the identity map on $L(\Z_n, \C^n)$ provides a Banach algebra isomorphism from $B(\Sigma_n)$ onto $B(\Omega_n)$. If one hopes that the order structure induced by the positive definite cone makes the Fourier-Stieltjes algebra into a more rigid invariant, it is cogent to ask if there exists an isomorphism that maps $P(\Sigma_n)$ onto $P(\Omega_n)$. It turns out that this possibility can be excluded already in the case $n = 2$. 

In order to describe the positive definite cones in $B(\Omega_2)$ and $B(\Sigma_2)$, it suffices to understand the cyclic equivariant representations of $\C^n$ on Hilbert $\C^n$-modules. This is summarized in the following result, which follows directly from \cref{thm:hilb_mods}. However, since our terminology here is somewhat non-standard, we include a sketch of the proof. 

Given $n$ Hilbert spaces $H_0, \ldots, H_{n-1}$, one can construct a Hilbert $\C^n$-module $X$ whose underlying vector space is $\bigoplus_{k \in \Z_n} H_k$, and where the action of $\C^n$ is given by $$(\xi_0, \ldots, \xi_{n-1}) \cdot (\lambda_0, \ldots, \lambda_{n-1}) = (\lambda_0 \xi_0, \ldots, \lambda_{n-1} \xi_{n-1}).$$ The inner product in $X$ is given by $$\langle (\xi_0, \ldots, \xi_{n-1}), (\eta_0, \ldots, \eta_{n-1}) \rangle = ( \langle \xi_0, \eta_0 \rangle, \ldots, \langle \xi_{n-1}, \eta_{n-1} \rangle ) $$  for $(\xi_0, \ldots, \xi_{n-1}), (\eta_0, \ldots, \eta_{n-1})\in X$. We will refer to $X$ as a \textit{sectional} Hilbert $\C^n$-module. By mild abuse of terminology, we also refer to the Hilbert spaces $H_0, \ldots, H_{n-1}$ as \textit{fibres}. 

\begin{proposition}
	\label{prop:Cn_rep}
	Suppose that $X$ is a Hilbert $\C^n$-module. Then $X$ is unitarily equivalent to a sectional Hilbert $\C^n$-module. Moreover, two sectional Hilbert $\C^n$-modules $X = \bigoplus_{k \in \Z_n} H_k$ and $Y= \bigoplus_{k \in \Z_n} K_k$ are unitarily equivalent if and only if the Hilbert spaces $H_k$ and $K_k$ are unitarily equivalent for each $k \in \Z_n$. 
	
	Every representation $\rho \colon \C^n \to \L(X)$ of $\C^n$ on a sectional Hilbert $\C^n$-module $X = \bigoplus_{k \in \Z_n} H_k$ decomposes as a direct sum of representations of $\pi_k \colon \C^n \to \B(H_k)$ for $k \in \Z_n$. The representation $\rho$ is cyclic if and only if $\pi_k$ is cyclic for every $k \in \Z_n$, in which case $\dim(H_k) \leq n$ for each $k \in \Z_n$. 
\end{proposition}

\begin{proof}
	Suppose that $X$ is a Hilbert $\C^n$-module. For each $k \in \Z_n$, set $H_k = X \cdot e_k = \{ x \cdot e_k : x \in X \}$. Then if $x \cdot e_k \in H_k$ and $y \cdot e_l \in H_l$ for $k, l \in \Z_n$, we have 
	\begin{equation} 
		\label{eq:inner_prod}
		\langle x \cdot e_k, y \cdot e_l \rangle = e_k \langle x, y \rangle e_l = \langle x, y \rangle e_k e_l  = \begin{cases} 
			\langle x, y \rangle e_k \quad & \text{if } k = l,\\
			0 \quad & \text{otherwise}.
		\end{cases}
	\end{equation} 
	It follows that, as a vector space, $X$ decomposes as the direct sum of $H_k$ for $k \in \Z_n$, and that there is a well-defined $\C$-valued inner product $\langle \cdot, \cdot \rangle_k$ on $H_k$ such that $\langle x \cdot e_k, y \cdot e_k \rangle_k$ is the $k$th component of $\langle x, y \rangle$. Completeness of each $H_k$ follows from completeness of $X$ and the fact that each $e_k$ is idempotent. Form the sectional $\C^n$-module $Y := \bigoplus_{k \in \Z_n} H_k$, and define a map $U \colon X \to Y$ by $$Ux = (x \cdot e_k)_{k \in \Z_n} \quad (x \in X).$$ It is routine to verify that $U$ defines a unitary map from $X$ to $Y$. In fact, \eqref{eq:inner_prod} is the essential observation.
	
	For the converse statement, suppose that $V \colon \bigoplus_{k \in \Z_n} H_k \to \bigoplus_{k \in \Z_n} K_k$ is a unitary operator between sectional Hilbert $\C^n$-modules. The fact that $V$ is $\C^n$-linear implies that $V$ is block-diagonal. Unitarity then implies that $V$ restricts to a unitary operator $v_k \colon H_k \to K_k$ for each $k \in \Z_n$. Similarly, if $\rho \colon \C^n \to \L(X)$ is a representation of $\C^n$ on a sectional Hilbert $\C^n$-module, $\C^n$-linearity again implies that each operator $\rho(a)$ is block-diagonal with respect to the given decomposition. The remaining assertions follow immediately from this.  
\end{proof}

\begin{example}
	Consider a cyclic equivariant representation $(\rho, v)$ of $\Omega_2$ on a Hilbert $\C^2$-module $X = H_0 \oplus H_1$. By \cref{prop:Cn_rep}, $\rho$ is block-diagonal, consisting of two representations $\pi_i \colon \C^2 \to \B(H_i)$ and $\dim H_0, \dim H_1 \leq 2$. Since the action is trivial, the equivariance relations imply that $v(1)$ is of the form $$v(1) = \begin{pmatrix}
		u_0 & 0 \\ 0 & u_1 
	\end{pmatrix}$$ for some unitary involutive operators $u_i \colon H_i \to H_i$ commuting with $\pi_i$ for $i=0,1$. For each $i \in \Z_2$, the dimension of $H_i$ may be $0, 1$ or $2$. If $\dim H_i = 2$, the representation $\pi_i$ must be equivalent to the representation of $\C^2$ on itself by diagonal matrices, in which case the commutant consists of the diagonal matrices. If $\dim H_i = 1$, the same is of course true, and if $\dim H_i = 0$, we have $u_i = 0$. By possibly embedding $H_0$ and $H_1$ into $\C^2$, we may assume that $v(1)$ is of the form $$v(1) = \begin{pmatrix}
		\varepsilon_0 & 0 & 0 & 0 \\
		0 & \varepsilon_1 & 0 & 0 \\
		0 & 0 & \varepsilon_2 & 0 \\
		0 & 0 & 0 & \varepsilon_3
	\end{pmatrix}$$ for $\varepsilon_i \in \{-1, 0, 1\}.$ Let $T$ denote the $\Omega_2$-positive definite multiplier associated to $(\rho, v)$ and a vector $(\xi, \eta) \in H_0 \oplus H_1 $. The standard matrices for $T_0$ and $T_1$ are given by \begin{equation}
		\label{eq:triv_pos_def}  \begin{pmatrix}
			\abs{\xi_0}^2 &  \abs{\xi_1}^2 \\
			\abs{\eta_0}^2 &\abs{\eta_1}^2 
		\end{pmatrix} \quad \text{and} \quad \begin{pmatrix}
			\varepsilon_0 \abs{\xi_0}^2 & \varepsilon_1 \abs{\xi_1}^2 \\
			\varepsilon_2 \abs{\eta_0}^2 & \varepsilon_3 \abs{\eta_1}^2 
		\end{pmatrix}
	\end{equation} 
	respectively.
\end{example}

\begin{example}
	The analysis of $P(\Sigma_2)$ follows the same strategy as that of $P(\Omega_2)$, but the equivariance relations are now different. Similarly to the previous case, if $(\rho, v)$ is a cyclic equivariant representation of $\Sigma_2$ on a sectional Hilbert $\C^2$-module $X = H_0 \oplus H_1$, $\rho$ consists of two representations $\pi_i \colon \C^2 \to \B(H_i)$ for $i \in \Z_2$ with $\dim(H_i) \leq 2$. The diligent reader is invited to verify that the equivariance relations imply that $v(1)$ is of the form $$v(1) = \begin{pmatrix}
		0 & u \\ u^* & 0 \end{pmatrix}$$ 
	for some unitary operator $u \colon H_{0} \to H_1$ such that 
	\begin{align}
		\label{eq:comm1}
		\pi_0(\alpha_1(a)) u &= u \pi_0(a) \\
		\label{eq:comm2}
		\pi_1(\alpha_1(a))u^* &= u^* \pi_1(a)
	\end{align} for $a \in \C^2$. In particular, $\dim H_0 = \dim H_1$. As above, we need only consider the case $\dim(H_0) = \dim(H_1) = 2$ to obtain a full description. In this case, each $\pi_i$ is unitarily equivalent to the representation of $\C^2$ on itself by diagonal matrices. Combining this with \eqref{eq:comm1} and \eqref{eq:comm2} gives that $$u = \begin{pmatrix}
		0 & \varepsilon_0 \\ \varepsilon_1 & 0 
	\end{pmatrix}$$ for some $\varepsilon_0, \varepsilon_1 \in \{-1, 1\}$. Using this, one can check that the $\Sigma_2$-positive definite multiplier associated to $(\rho, v)$ and a vector $(\xi, \eta) \in H_0 \oplus H_1 = X$ has standard matrices 
	\begin{equation} 
		\label{eq:cyc_mat} 
		\begin{pmatrix}
			\abs{\xi_0}^2 & \abs{\xi_1}^2 \\
			\abs{\eta_0}^2 & \abs{\eta_1}^2
		\end{pmatrix} \quad \text{and} \quad \begin{pmatrix}
			\varepsilon_0\overline{\xi_0} \eta_1 & \varepsilon_1 \overline{\xi_1}\eta_0 \\
			\varepsilon_0 \overline{\eta_0} \xi_1 & \varepsilon_1 \overline{\eta_1} \xi_0
		\end{pmatrix}.
	\end{equation} 
	Notice the form of the second matrix compared to the previous example. 
\end{example}

With these concrete descriptions of $P(\Sigma_2)$ and $P(\Omega_2)$, it is not difficult to see that no isomorphism between $B(\Sigma_2)$ and $B(\Omega_2)$ can carry one to the other. Indeed, since every (algebraic) isomorphism of $M_2(\C)$ is inner, see \cite[Theorem 20']{lorenz}, every isomorphism of $M_2(\C) \oplus M_2(\C)$ must be of the form $\left( \mathrm{Ad}(V) \oplus \mathrm{Ad}(W) \right) \circ F^\varepsilon$ for some invertible matrices $V, W \in GL_2(\C)$ and $\varepsilon \in \{0, 1\}$, where $F$ denotes the flip on $M_2(\C) \oplus M_2(\C)$. In particular, every automorphism commutes with the componentwise trace, at least up to a flip of the factors. However, by using \cref{eq:triv_pos_def}, the image of $P(\Omega_2)$ under the componentwise trace is $[0, \infty) \times \R$. On the other hand, using \eqref{eq:cyc_mat} the image of $P(\Sigma_2)$ under the componentwise trace is $\R \times \C$. Hence, no such isomorphism can exist. 

\section{Concluding remarks and open questions}

We end by pointing to some interesting open questions relating to the present work, some of which are also mentioned in \cite{skalski} in their context. 

With two definitions of the Fourier algebra on the table, the most pressing question seems to be to compare the two. There are two concrete questions that are relevant in this regard: 

\begin{question}[{cf. \cite[Question 12.1]{skalski}}] Is the norm defined for coefficient functions of regular equivariant representations in \cite{skalski} equal/equivalent to the norm in the Fourier-Stieltjes algebra? 
\end{question}

\begin{question}[{cf. \cite[Question 12.3]{skalski}}]
	Is the set of coefficient functions of regular equivariant representation closed in the Fourier-Stieltjes algebra? 
\end{question} 

Note that if the second question can be answered in the affirmative, the two norms must be equivalent by the Open mapping theorem since Buss et al. show in \cite{skalski} that the set of coefficients of regular equivariant representations is complete with respect to their norm. Conversely, if the norms are equivalent, it follows that our definitions of the Fourier algebra coincide as algebras. 

It is also interesting to study approximation properties for dynamical systems in terms of positive definite multipliers with or without finite support. In their preprint \cite{skalski}, Buss et al. devote Section 10 to such questions. An interesting problem which they do not address is how such approximation properties relate to the existence of a bounded approximate unit in the Fourier algebra. A classical theorem due to Leptin says that a group is amenable if and only if its Fourier algebra admits a bounded approximate unit, in which case it can be assumed to consist of positive definite functions with compact support, see \cite[Theorem 2.7.2]{kaniuth}. It is an easy exercise to show that if the Fourier algebra admits a bounded approximate unit whose elements are finitely supported positive definite multipliers, the underlying system is amenable in the sense of Bédos and Conti, see \cite[Definition 4.4]{BC2016}. By density, it is also clear that any bounded approximate unit can be assumed to consist of multipliers with finite support. However, making these multipliers positive definite in a way that preserves all the desired properties seems non-trivial. Answering the following question would be a first step in the direction of a generalization of Leptin's theorem. 

\begin{question}
	Can a bounded approximate unit in the Fourier algebra be transformed to one consisting of positive definite multipliers (with finite support)? 
\end{question}

\subsection*{Acknowledgements} 
	The present article is based on the author's master's thesis \cite{ravnanger}, written under the supervision of Erik Bédos at the University of Oslo during the fall of 2023 and spring of 2024. The author is grateful to Erik Bédos for excellent guidance during the thesis work and for many helpful comments upon reading early versions of this manuscript. Thanks also to the referee for several helpful comments. We note that the preprint \cite{skalski} was published after the submission of said thesis.

\printbibliography{}
	
\end{document}